\pdfoutput=1

\documentclass{article}
	\usepackage{amsmath, amsfonts, amssymb, amsthm, graphicx, mathrsfs, bm}

\title{Cyclotomic torsion points in elliptic schemes}
\author{Michele Giacomini}
\date{}

\usepackage[utf8]{inputenc}
\usepackage[T1]{fontenc}
\usepackage{lmodern}
\usepackage[english]{babel}	

\usepackage{graphicx}
\graphicspath{{./Figure/}}

\usepackage[font={small,it}, labelfont=bf]{caption} 
\usepackage{subcaption}
\usepackage[nottoc]{tocbibind}
\usepackage{comment}
\usepackage{xspace}

\usepackage{xfrac}
\usepackage[usenames,dvipsnames]{color}
\usepackage{enumerate}
\usepackage[top=3cm, bottom=3cm, left=2.5cm, right=2.5cm, bindingoffset=1cm]{geometry}
\usepackage{csquotes}

\usepackage[backend=biber,%
style=alphabetic,%
sorting=nyt,maxnames=5,maxalphanames=5,%
firstinits=true,    
isbn=false, doi=false, url=false, eprint=false %
]{biblatex}

\AtEveryBibitem{%
\clearfield{edition}
\clearlist{location}
}

\usepackage{tikz}
\usetikzlibrary{external}
\usetikzlibrary{bending}
\usetikzlibrary{arrows.meta}

\newcommand{\G}{\mathbb{G}}
\newcommand{\A}{\mathbb{A}}
\newcommand{\R}{\mathbb{R}}
\newcommand{\C}{\mathbb{C}}
\newcommand{\N}{\mathbb{N}}
\newcommand{\Z}{\mathbb{Z}}
\newcommand{\Q}{\mathbb{Q}}

\newcommand{\abs}[1]{\left\vert #1\right\vert}
\newcommand{\norm}[1]{\left\lVert #1\right\rVert}
\newcommand{\de}{\partial}
\newcommand{\frecciadx}[2]{\mathrel{\mathop{\longrightarrow}^{\mathrm{#1}}_{\mathrm{#2}}}}

\renewcommand{\bar}[1]{\overline{#1}}
\renewcommand{\tilde}[1]{\widetilde{#1}}

\DeclareMathOperator{\alg}{alg}
\DeclareMathOperator{\an}{an}
\DeclareMathOperator{\trdeg}{trdeg}

\makeatletter
\newcommand*{\house}[1]{%
  \mathord{%
    \mathpalette\@house{#1}%
  }%
}
\newcommand*{\@house}[2]{%
  \dimen@=\fontdimen8 %
      \ifx#1\scriptscriptstyle\scriptscriptfont
      \else\ifx#1\scriptstyle\scriptfont
      \else\textfont\fi\fi
      3 %
  \sbox0{%
    $#1%
      \vrule width\dimen@\relax
      \overline{%
        \kern2\dimen@
        \begingroup 
          #2%
        \endgroup
        \kern2\dimen@
      }%
      \vrule width\dimen@\relax
      \mathsurround=1.5\dimen@ 
    $%
  }%
  \ht0=\dimexpr\ht0-\dimen@\relax
  \dp0=\dimexpr\dp0+2\dimen@\relax
  \vbox{%
    \kern\dimen@ 
    \copy0 %
  }%
}
\makeatother

\theoremstyle{definition}
\newtheorem{definition}{Definition}[section]
\theoremstyle{plain}
\newtheorem{theorem}[definition]{Theorem}

\newtheorem{lemma}[definition]{Lemma}
\newtheorem{proposition}[definition]{Proposition}
\newtheorem{corollary}[definition]{Corollary}

\theoremstyle{remark}
\newtheorem{remark}[definition]{Remark}
\newtheorem{example}[definition]{Example}

\DefineBibliographyStrings{english}{%
  bibliography = {References},
}

\begin{document}

\maketitle

\begin{abstract}
  An elliptic curve defined over a number field possesses only a finite number of torsion points defined over the cyclotomic closure of its field of definition. In analogy to the relative version of the Manin-Mumford conjecture stated by Masser and Zannier, we propose a family version of the above statement and prove it under a suitable integrality condition.
\end{abstract}

\tableofcontents


\section{Introduction} 
\label{sec:introduction}

First of all we recall a known result about torsion points in elliptic curves, this follows form the famous theorem of Serre about Galois representations attached to elliptic curves \cite{Serre72}.

\begin{theorem}
\label{teo:cyclotomic_torsion_points_elliptic_curve}
	Let $E$ be an elliptic curve defined over a number field $k$. Then there are only a finite number of torsion points of $E$ which are defined over the cyclotomic closure $k^c$ of $k$\footnote{Let $\bar{k}$ be an algebraic closure of $k$, then $k^{c}$ is the field obtained from $k$ by adjoining all roots of unity contained in $\bar{k}$.}.
\end{theorem}

In analogy to the relative formulation of the Manin-Mumford conjecture, proven by Masser and Zannier for products of elliptic schemes in \cite{MasserZannier10,MasserZannier12,MasserZannier14}; we are prompted to formulate a relative version of the previous result. That is, we expect that given an elliptic scheme $\mathcal{E}/\mathcal{C}$ and a non-torsion section $s:\mathcal{C}\to \mathcal{E}$ both defined over a number field $k$, there are only finitely many points $P\in \mathcal{C}(k^c)$ such that $s(P)\in \mathcal{E}_{P}$ is torsion, where $k^c$ is a cyclotomic closure of $k$. This finiteness expectation is supported by the fact that elliptic curves and the multiplicative group $\mathbb{G}_{m}$ are essentially unrelated; for example there are no non constant algebraic maps between an elliptic curve and $\mathbb{G}_{m}$. This suggests that it is unlikely that the field of definition for the torsion points of the multiplicative group, i.e. $\mathbb{Q}^{c}$, contains the coordinates of torsion points in an elliptic curve.

One of the aims of this paper is to solve this problem under a suitable integrality condition.

\begin{theorem}[(Main Result)]
\label{teo:main_result}
	Let $k$ be a number field and $C/k$ a curve. Let $\mathcal{E}/C$ be an elliptic scheme and let $s:C\to \mathcal{E}$ be a section which is not constantly torsion, both defined over $k$. Let $f:C\to \mathbb{A}^1$ be a non constant rational function defined over $k$. Finally, let $n$ be a natural number. Then there are only finitely many points $P\in C$ such that $f(P)$ is the sum of $n$ roots of unity and such that $s(P)\in\mathcal{E}_{P}$ is torsion.
\end{theorem}

We recall that the algebraic integers in $\mathbb{Q}^{c}$ are exactly the complex numbers that can be written as sum of roots of unity. Hence our main result indeed proves the expectation expressed above for a particular set of algebraic integers.

The proof follows the general lines of the one used by Pila and Zannier to prove the Manin-Mumford conjecture in \cite{PilaZannier08}, and later exploited by Masser and Zannier to analyse problems concerning relative Manin-Mumford. In particular, we follow closely the isotrivial case in the proof of relative Manin-Mumford for an extension of an elliptic curve by $\G_m$, given by \citeauthor{BertrandMasserPillayZannier16} in the paper \cite{BertrandMasserPillayZannier16}. Nevertheless, our case presents some issues which did not appear in the above cited case. The most important one is the presence of semialgebraic sets of dimension $\geq 1$ in the transcendental variety defined using logarithms.

Our main result constitutes a generalisation of the isoconstant case in the paper \cite{BertrandMasserPillayZannier16}. We now demonstrate a simple consequence of this fact; using the results of the above cited paper, one can prove the following proposition.

\begin{proposition}
\label{cor:matrici}
	Let $\mathcal{E}/\A^1$ be an elliptic scheme and let $s:\A^1\to \mathcal{E}$ be an algebraic section which is not constantly torsion, both defined over a number field $k$. Then there are only finitely many $\lambda\in \A^1$ such that 
	\begin{equation}
		\left(s(\lambda),\begin{pmatrix}0 & 1 \\ -1 & \lambda\end{pmatrix}\right)
	\end{equation}
	is torsion in $\mathcal{E}_{\lambda}\times SL_2(\C)$.
\end{proposition}

We observe that a matrix of the form $\left( \begin{smallmatrix}
	0 & 1 \\ -1 & \lambda
\end{smallmatrix} \right) $ is torsion in $SL_2(\C)$ if and only if $\lambda$ is sum of two roots of unity and is real. Using Theorem \ref{teo:main_result}, we can generalise this result, removing the assumption that $\lambda$ be real and allowing $\lambda$ to be an arbitrary sum of a fixed number of roots of unity, obtaining again finiteness.

\paragraph{Acknowledgements} 
\label{par:acknowledgements}

This article is the result of the work for my master's dissertation sustained at the University of Udine. I would like to thank my supervisor Professor Pietro Corvaja for introducing me to this topic, pointing me to the problem and all the invaluable help and discussions throughout the preparation of this paper.

\section{o-minimal structures and the Pila-Wilkie theorem} 
\label{sec:o_minimal_structures_and_the_pila_wilkie_theorem}

In this section we recall the definition of o-minimal structure and the statement of Pila-Wilkie's theorem on rational points in definable sets.

First we recall the definition of semialgebraic sets.

\begin{definition}
	\label{def:semialgebraic set}
	A subset $X$ of $\R^n$ is said to be \emph{semialgebraic} if it is a finite union of sets of the form:
	\begin{equation}
		\{\bm{x}\in\R^n | p_1(\bm{x})=\ldots=p_r(\bm{x})=0\text{ and }q_1(\bm{x}),\ldots,q_{s}(\bm{x})>0\},
	\end{equation}
	where $p_1\ldots,p_r,q_1,\ldots,q_s$ are polynomials.
\end{definition}

\begin{definition}
	\label{def:o-minimal structure}
	A \emph{structure} over $\mathbb{R}$ is a collection $\mathcal{S}=\{\mathcal{S}_n| n\in\N\}$ such that:
	\begin{enumerate}
		\item $\mathcal{S}_n$ is a boolean algebra of subsets of $\mathbb{R}^n$, i.e. if $S,S'\in \mathcal{S}$, $S\bigcup S'\in \mathcal{S}_n$ and $\mathbb{R}^n\setminus S\in \mathcal{S}_n$;
		\item if $S\in \mathcal{S}_n$ and $S'\in \mathcal{S}_m$, then $S\times S' \in \mathcal{S}_{n+m}$;
		\item if $A\in \mathcal{S}_{n+1}$, then $\pi(A)\in \mathcal{S}_n$, where $\pi:\mathbb{R}^{n+1}\to \mathbb{R}^n$ is the projection on the first $n$ coordinates;
		\item $\mathcal{S}_{n}$ contains all semialgebraic subsets of $\mathbb{R}^{n}$.
	\end{enumerate}
	The structure is called \emph{o-minimal} if, in addition to the conditions above, 
	\begin{enumerate}
		\setcounter{enumi}{4}
		\item The sets in $\mathcal{S}_1$ are exactly the finite unions of open intervals and points.
	\end{enumerate}
	We say that a set $A\subset \mathbb{R}^n$ is \emph{definable} if it belongs to $\mathcal{S}_n$; moreover a map $f:\mathbb{R}^n\to \mathbb{R}^m$ is \emph{definable} if its graph $\Gamma(f)$ belongs to $\mathcal{S}_{n+m}$.
\end{definition}

We now introduce the two examples of o-minimal structures we will be using in what follows.

\begin{example}[(Semialgebraic sets)]
	The collection $\mathcal{S}=\{\mathcal{S}_n| n\in\N\}$ where $\mathcal{S}_{n}$ is the set of semialgebraic sets of $\mathbb{R}^{n}$ is an o-minimal structure, thanks to Tarski's theorem on elimination of quantifiers (see \cite{Tarski48}).

	From the definition of o-minimal structure, one sees that this is the smallest possible o-minimal structure.
\end{example}

\begin{example}[($\mathbb{R}_{an,exp}$)]
	$\mathbb{R}_{an,exp}$ is the smallest o-minimal structure such that $\mathcal{S}_{2}$ contains the graph of the exponential function and, for each $n$, $\mathcal{S}_{n+1}$ contains the graph of all restricted analytic functions $f:\mathbb{R}^{n}\to \mathbb{R}$. Recall that the function $f$ is called restricted analytic if it is zero outside $[-1,1]^{n}$ and on $[-1,1]^{n}$ coincides with an analytic function defined on an open neighbourhood $U$ of $[-1,1]^{n}$. See \cite{DriesMacintyreMarker94} for the result stating that such an o-minimal structure exists.
\end{example}

From now on, we will always use definable to mean definable in the structure $\mathbb{R}_{an,exp}$.

We now recall some more definitions before stating Pila-Wilkie's Theorem.

\begin{definition}
	Let $H:\Q\to \R$ be the usual height function, $H\left(\frac{a}{b}\right)=\max\{\abs{a},\abs{b}\}$, where $a,b\in\Z$ and $(a,b)=1$. We extend $H$ to $\Q^n$ setting $H(x_1,\ldots,x_n)=\max_{i}H(x_i)$. If $X\subset \R^n$ we denote by $X(\Q)$ the subset of points of $X$ with rational coordinates. If $T\geq 1$ we define
	\begin{equation}
		X(\Q,T)=\{P\in X(\Q)| H(P)\leq T\}.
	\end{equation}
	Finally we define
	\begin{equation}
		N(X,T)=\#X(\Q,T).
	\end{equation}
\end{definition}

\begin{definition}
	Let $X\subset \R^n$. The \emph{algebraic part} of $X$, denoted by $X^{\alg}$, is the union of all connected semialgebraic subsets of positive dimension of $X$.
\end{definition}

\begin{theorem}[{\cite[Theorem 1.8]{PilaWilkie06}}]
	\label{thm:Pila-Wilkie}
	Let $X\subset \R^n$ be a definable set. Let $\varepsilon>0$. Then there exists a number $c(X,\varepsilon)>0$ such that
	\begin{equation}
		N(X\setminus X^{\alg},T)\leq c(X,\varepsilon)T^{\varepsilon}.
	\end{equation}
\end{theorem}



\section{Proof of the Main Result} 
\label{sec:proof_of_the_main_result}

First we recall the statement of the main result.

\begin{theorem}
	\label{thm:main result}
	Let $k$ be a number field and $C/k$ a curve. Let $\mathcal{E}/C$ be an elliptic scheme and let $s:C\to \mathcal{E}$ be a section which is not constantly torsion, both defined over $k$. Let $f:C\to \mathbb{A}^1$ be a non constant rational function defined over $k$. Finally, let $n$ be a natural number. Then there are only finitely many points $P\in C$ such that $f(P)$ is the sum of $n$ roots of unity and such that $s(P)\in\mathcal{E}_{P}$ is torsion.
\end{theorem}

First we observe that, up to removing a finite number of points from $C$, we may assume that it is an affine non singular curve and the function $f:C\to \mathbb{A}^{1}$ is regular. 

The basic idea of the proof strategy is to construct a set definable in some o-minimal structure whose rational points correspond to the points we are interested in; then to compare the upper bound on the number of this points given by  Pila-Wilkie's Theorem \ref{thm:Pila-Wilkie} with a lower bound coming from Galois theoretic information.

Now, we fix notation for the following sections.
\begin{itemize}
	\item We keep the same notation as in Theorem \ref{thm:main result}. Namely, $k$ is a number field, $C$ a curve, which as above we may assume affine and non singular, $\mathcal{E}/C$ is an elliptic scheme, $s:\A^1\to \mathcal{E}$ is a section which is not constantly torsion and $f:C\to \mathbb{A}^1$ is a regular function, all defined over $k$. Finally, $n$ is a fixed natural number.
	\item $P$ will always denote a point in $C$.
	\item $\bm{x}=(x_1,\ldots,x_n)$ and $\bm{\varepsilon}=(\varepsilon_1,\ldots,\varepsilon_n)$ will denote a points in $\G_m^n$ and $\bm{\varepsilon}$ will always be assumed torsion.
	\item By $c_1,c_2,\ldots$ we will denote positive real numbers depending on the data of the theorem, that is: the elliptic scheme $\mathcal{E}/C$, the section $s$, the function $f$ and the natural number $n$. 
	\item Given a scheme $\mathcal{S}$ we will denote by $\mathcal{S}^{\an}$ the analytic space defined by it.
	\item Fix an embedding $C\subset \mathbb{A}^{N}$ for some $N$. $H$ is the absolute height function on $C$ associated to this affine embedding.
	\item We define the map 
	\begin{equation}
		\Sigma:\G_m^n\to \mathbb{A}^1
	\end{equation}
	as $\bm{x}\mapsto \sum_{i=1}^{n}x_{i}.$
	\item Consider the fibre products $C\times_{\mathbb{A}^{1}}\mathbb{G}_{m}^{n}$ and $\mathcal{E}\times_{\mathbb{A}^{1}}\mathbb{G}_{m}^{n}$ taken with respect to the maps $f$ and $\Sigma$. Define the map 
	\begin{equation}
	\sigma:C\times_{\mathbb{A}^{1}}\mathbb{G}_{m}^{n}\to \mathcal{E}\times_{\mathbb{A}^{1}}\mathbb{G}_{m}^{n}	
	\end{equation}
	to be the section $s$ on the first component and the identity in the $\mathbb{G}_{m}^{n}$ component.
\end{itemize}

\subsection{Algebraic Lower Bounds} 
\label{sub:algebraic_lower_bounds}

The results of this section are consequences of the work done by Masser and Zannier on the subject. We keep notation as above.

Te following Proposition is a compound of result proven by \citeauthor{MasserZannier12}, see for example \cite{MasserZannier12} or \cite{Zannier12}, and is a consequence of work by Silverman and David.

\begin{proposition}[Proposition 1, Section 3.1 in \cite{ BertrandMasserPillayZannier16}]
\label{prop:punti_torsione_schema_ellittico}
	Let $\mathcal{E}/C$ be an elliptic scheme and $s:C\to \mathcal{E}$ a section, which is not constantly torsion, both defined over a number field $k$. Then there are two positive real numbers $c_1$ and $c_2$, depending only on $C/k$, $\mathcal{E}/C$ and $s$, with the following properties. Let $P\in C$ be a point such that $s(P)\in \mathcal{E}_{P}$ is torsion, then:
	\begin{enumerate}
		\item $P$ is algebraic;
		\item the height $H(P)$ is bounded from above by $c_1$;
		\item if $T$ is the order of $s(P)$, then $[k(P):k]\geq c_2 T^{1/3}$.
	\end{enumerate}
\end{proposition}

Now we add the multiplicative component we will need in the sequel.

\begin{corollary}
\label{cor:punti_torsione_EGm2}
	With notation as above, let $\bm{\varepsilon}\in \G_m^n$ and $P\in C$ such that $s(P)\in\mathcal{E}_{P}$ and $\bm{\varepsilon}$ are torsion. Let $T$ be the order of $(\bm{\varepsilon},s(P))$ in $\G_m^n \times \mathcal{E}_{P}$, then there is a positive real number $c_1$, depending only on $C/k$, $\mathcal{E}/C$, $s$ and $n$, such that $[k(\bm{\varepsilon},P):k]\geq c_1 T^{1/6}$.
\end{corollary}

\begin{proof}
	Let $h$ be the order of $s(P)\in \mathcal{E}_{P}$, then $\bm{\varepsilon}^h$ has order $T/h$. We have that
	\begin{equation}
		[k(\bm{\varepsilon}):k]\geq \frac{1}{[k:\Q]}\phi \left( \frac{T}{h} \right) \geq \frac{1}{[k:\Q]\sqrt{2}}\sqrt{\frac{T}{h}}=c_3 \left( \frac{T}{h}\right)^{\frac{1}{2}},
	\end{equation}
	Where $\phi$ is Euler totient function and we have used $\phi(x)\geq \sqrt{\frac{x}{2}}$.
	Moreover from the previous proposition, we have that there is a number $c_4$ such that $[k(P):k]\geq c_4 h^{1/3}$. Putting $c_1=\min \left\{ c_3,c_4 \right\}$, we calculate
	\begin{equation}
	\begin{split}
			\left[ k(\bm{\varepsilon},P):k \right] %
			&\geq \max \left\{ [ k(\varepsilon^h):k ], \left[ k(P):k \right]   \right\}%
			\geq \max \left\{ c_3 \left( \frac{T}{h} \right)^{\frac{1}{2}}, c_4 h^{\frac{1}{3}}  \right\} \\
			&\geq c_1 \max \left\{ \frac{T}{h}, h \right\}^{\frac{1}{3}} %
			\geq c_1 T^{\frac{1}{6}}. 
	\end{split}
	\end{equation}
\end{proof}

\begin{corollary}
\label{cor:grado_punti_torsione}
	With notation as above, there are two positive real numbers $c_1$ and $c_2$, depending only on $C/k$, $\mathcal{E}/C$, $s$ and $n$, such that, for all $(P,\bm{\varepsilon})\in C\times\G_m^n$ with $\sigma(P,\bm{\varepsilon})\in \mathcal{E}_{P}\times\G_m^n$ torsion of order $T$, the following bound holds:
	\begin{equation*}
		[k(P,\bm{\varepsilon}):k]\geq c_1 T^{c_2}.
	\end{equation*}
\end{corollary}

As it is, the Galois theoretic information contained in the last corollary is not sufficient. We will need to construct a subset $S$ of $(C\times_{\mathbb{A}^{1}}\mathbb{G}_{m}^{n})(\mathbb{C})$, compact in the euclidean topology, with the following property: $S$ avoids the points of bad reduction for the scheme $\mathcal{E}\times_{\mathbb{A}^{1}}\mathbb{G}_{m}^{n}/C\times_{\mathbb{A}}^{1}\G_m^n$ and such that, for any point of $(P,\bm{x})\in \mathcal{E}\times_{\mathbb{A}^{1}}\G_m^n$ such that $\sigma(P,\bm{x})$ is torsion, a fixed fraction of the conjugates of $\bm{x}$ over $k$ lies in $S$. 

The existence of such a compact set will allow us to construct the definable set to which we will apply Pila-Wilkie's Theorem.

The main ingredient in the construction of the compact set $S$ is the fact that the points $P\in C$ such that $s(P)$ is torsion in $\mathcal{E}_{P}$ are of bounded height. We will construct $S$ separately on $C$ and $\G_m^n$. 

First we analyse the $C$ part in the following proposition, which is a slight generalisation of \cite[Lemma 8.2]{MasserZannier12}.

\begin{proposition}
\label{prop:Kdelta1}
	Let $\mathcal{E}/C$ be an elliptic scheme over an affine non singular curve $C\subset \mathbb{A}^{N}$, both defined over a number field $k$. Let $B\subset C$ be a finite set set of points containing the points of bad reduction of $\mathcal{E}/C$. Fix a positive constant $a$ and a number field $K$ containing $k$ and a field of definition of $B$. Then, there is some positive real number $\delta$ such that the compact set 
	\begin{equation}
		\mathscr{K}_{\delta}^{1}:=\{P\in C\subset \mathbb{A}^{N}: \norm{P}\leq 1/\delta,\ \norm{P-\beta}\geq \delta\ \forall \beta\in B\}
	\end{equation}
	has the following property: for all $P\in C\setminus B$ such that $H(P)\leq a$, at least $1/2$ of the $[K(P):K]$ conjugates of $P$ over $K$ are contained in $\mathscr{K}_{\delta}^{1}$.
\end{proposition}
\begin{proof}
	Let $l=\# B+1$, that is, $l$ is the number of inequalities defining the set $\mathscr{K}_{\delta}^{1}$. Fix $\beta\in B$ and let $\mathscr{I}$ be the set of embeddings $i$ of $K(P)$ into $\mathbb{C}$ sending $K$ to itself such that $\norm{i(P)-\beta}<\delta$. Then we have
	\begin{equation}
		h(P- \beta)\leq a+h(\beta)+\log(2).
	\end{equation}
	On the other hand, write $P=(p_1,\ldots,p_N)$ and $\beta=(\beta_{1},\ldots,\beta_{N})$ and let $1\leq j\leq N$ be such that $h(p_{j}-\beta_{j})=\max_{h=1,\ldots, m}h(p_{h}-\beta_{h})$. Then
	\begin{equation}
	\begin{split}
		h(P-\beta )\geq h(p_{j}-\beta_{j})
		&\geq \frac{1}{d^{*}}\sum_{v}\log\max \left\{ 1,\abs{\frac{1}{p_{j}-b_{j}}}_{v} \right\}\\
		&\geq \frac{1}{d^{*}}\sum_{i\in \mathscr{I}} \log\max \left\{ 1,\abs{\frac{1}{i(p_{j})-b_{j}}}\right\}\\
		&\geq \frac{\abs{\mathscr{I}} }{d^{*}}\log \left( \frac{1}{\delta} \right);
	\end{split}
	\end{equation}
	where $d^{*}=[K(P):\mathbb{Q}]$. From the above inequalities we get
	\begin{equation}
		\abs{\mathscr{I} }\leq \frac{\left(a+h(\beta)+\log(2)\right)d^{*}}{\log(\frac{1}{\delta})}. 
	\end{equation}
	Hence we may choose $\delta$ such that
	\begin{equation}
		\abs{\mathscr{I}}\leq \frac{d^{*}}{2l [K:\mathbb{Q}]}=\frac{[K(P):K]}{2l}. 
	\end{equation}
	A similar proof works for the inequality $\norm{P}\leq 1/\delta$. 
\end{proof}

Now we turn to the $\G_m^n$ part. This is easier since all we need here is to have some compact subset of $\mathbb{G}_{m}^{n}(\mathbb{C})$ whose interior in the euclidean topology contains all its torsion points. Thus, it is sufficient to consider the compact set $\mathscr{K}^2=\{\bm{x}\in \G_m^n | 1/2\leq|x_i|\leq3/2,\ i=1,\ldots,n \}$. 

We can now give the definition of the compact set $S$.

\begin{definition}
\label{def:compatto_S}
	Let $c_1$ be the number given by Proposition \ref{prop:punti_torsione_schema_ellittico} relative to the scheme $\mathcal{E}/C$; moreover let $\delta$ be the number given by Proposition \ref{prop:Kdelta1} relative to $a=c_1$, $B$ the set of points of bad reduction for $\mathcal{E}/C$ and $K$ the number field of definition for $B$. We define 
	\begin{equation}
	S=\mathscr{K}_{\delta}^{1}\times_{\mathbb{A}^{1}}\mathscr{K}^{2} .	
	\end{equation}
\end{definition}

The next proposition is then clear from the results proven in this section and the definition of $S$.

\begin{proposition}
	\label{prop:properties S}
	$S$ is a semialgebraic set of dimension $2n$ and, for all $(P,\bm{\varepsilon}) \in C\times_{\mathbb{A}^{1}}\mathbb{G}_{m}^{n}$ such that $\sigma(P, \bm{\varepsilon} )$ is torsion in $\mathcal{E}_{P}\times \mathbb{G}_{m}^{n}$, at least half of the $[K(P, \bm{\varepsilon}):K]$ conjugates of $(P, \bm{\varepsilon})$ over $K$ lie in $S$.
\end{proposition}


\subsection{Construction of the Logarithm} 
\label{sub:construction_of_the_logarithm}

The aim of this section is to construct the definable set to which we will apply Pila-Wilkie's Theorem \ref{thm:Pila-Wilkie}.

First we observe that, after removing the bad fibres, $\mathcal{E}^{\an}\to C^{an}$ gives an analytic family of compact complex Lie groups over $C^{an}$. Moreover we have that the relative Lie algebra $Lie(\mathcal{E})/C$  gives us an analytic vector bundle $Lie(\mathcal{E})^{\an}\to C^{an}$ of rank $1$ over $C^{an}$. With this bundle we are also given the exponential map $\exp:Lie(\mathcal{E})^{\an}\to \mathcal{E}^{\an}$, which is an analytic and surjective map of varieties over $\C$. The kernel $\Pi_{\mathcal{E}}$ of $\exp$ is a locally constant sheaf of rank $2$ $\mathbb{Z}$-modules over $C^{an}$, which we call the \emph{local system of periods} of $\mathcal{E}/C$. 

Now, let $W$ be an open contractible subset of $C^{an}$, so that both the local system of periods $\Pi_{\mathcal{E}}$ and the line bundle $Lie(\mathcal{E})^{an}/C$ are trivial over $W$\footnote{
	We note that $C^{an}\subset \mathbb{C}^{N}$ is a Stein space, this implies that all holomorphic line bundles over it are trivial. Hence the bundle $Lie(\mathcal{E})^{an}\to C^{an}$ is already trivial and we need to choose and open contractible set to be sure that the local system of periods is also trivial.}.
Choose a section $W\to W\times \mathbb{C}\cong Lie(\mathcal{E})^{\an}|_{W} $ of $\exp$ over $W$ and denote by ${\log}_{W}$ the composition  $W\to W\times \mathbb{C}\to \mathbb{C}$. Then, we define ${\log}_{E,U}$ as the composition ${\log}_{U}\circ s$. 

Moreover, for any such open subset $W$ of $C^{an}$, we have that the sections of $\Pi_{\mathcal{E}}$ over $W$ form a free $\Z$-module of rank $2$. We fix a basis for this $\Z$-module and denote it by ${\omega}_{1,W},{\omega}_{2,W}$. 

The \emph{Betti presentation of the logarithm} ${\log}_{E,W}$ is defined as the pair of real analytic functions $({b}_{1,W},{b}_{2,W}):W\to \R^2$ such that ${\log}_{E,W}={b}_{1,W} {\omega}_{1,W} + {b}_{2,W} {\omega}_{2,W}$.

Further, let $W'$ be an open contractible subset of ${\mathbb{G}_{m}^{n}}^{an}={\mathbb{C}^{*}}^{n}$; define $p_j: \G_{m}^{n}\to \mathbb{G}_{m}$ to be the projections to the $j$-th  component, and let $\log_{j,W'}$ be a determination of the logarithm defined on the projection $p_{j}(W')$. Given this data, we set $a_{j,W'}=(\log_{j,W'}\circ p_{j})/2\pi i$.

We observe that, given a point $(P, \bm{x} )\in C\times_{\mathbb{A}^{1}} \mathbb{G}_{m}^{n}$, $\sigma(P, \bm{x} )$ is torsion if and only if the point $({b}_{1,W}(P), {b}_{2,W}(P), a_{1}(x), \ldots, a_{n}(x)) \in \mathbb{R}^{2} \times \mathbb{C}^{n}$ is rational.

See Figure \ref{fig:diagramma_mappe} for a diagram containing all the maps defined up to this point.

\begin{figure}[htb]
	\hspace{-20pt}
	\begin{center}

\begin{tikzpicture}[scale=2.3]
   			\node (W)		at		(0,0)		{$W$};
   			\node (Curve)	at 		(0.25,0) 	{$\subset C$};
   			\node (E)		at		(-1,1)		{$\mathcal{E}^{\an}\times_{\A^1}W$};
   			\node (Lie)		at		(-1,2)		{$\C\times W$};
   			\node (PE)		at		(-2,2)		{$\Pi_{\mathcal{E}}\times_{\A^1}W$};
   			\node (CE)		at		(-1,3)		{$\C$};
   			\node (R2)		at		(.5,3)		{$\R^2$};
			\node (W')		at		(1,0)		{$W'$};
			\node (GmN)		at 		(1.32,0)	{$\subset \mathbb{G}_{m}^{n}$};
			\node (Gm1)		at		(1.5,1)		{$p_{1}(W')$};
			\node (pti)		at		(2,1)		{$\cdots$};
			\node (Gmn)		at		(2.5,1)		{$p_{n}(W')$};
			\node (C1)		at		(1.5,2)		{$\C$};
			\node (Cn)		at		(2.5,2)		{$\C$};
			\node (A1)		at 		(.5,-.5)		{$\mathbb{A}^{1}$};
			\draw[-stealth] (W)		edge[bend left=20]	node[below]				{{$s$}}														(E)
							(W)		edge[out=170,in=180,absolute,looseness=1.85]	node[left,near start]	{{${\omega}_1,{\omega}_2$}}	(CE)		  
							(W)		edge[out=90, in=330,absolute]	node[right, near end]				{{${\log}_{E,W}$}}						(CE)
							(W)		edge[bend right=10]	node[right, near end]		{{$({b}_{1,W},{b}_{2,W})$}}		  (R2)
							(E)		edge				node[above]				{{$\pi$}}									(W)
							(E)		edge[bend right=30]	node[right]				{{${\log}_{W}$}}						(CE)
							(Lie)	edge				node[left]				{{$\exp$}}									(E)
							(Lie)	edge																								(CE)
							(W')		edge				node[near start, below]				{{$\Sigma$}}									(A1)
							(W)		edge				node[near start, below]				{{$f$}}									(A1)
							(W')		edge				node[right]				{{$p_1$}}									(Gm1)
							(W')		edge				node[right]				{{$p_n$}}									(Gmn)
							(Gm1)	edge				node[right]				{{$\frac{\log_{1,W'}}{2\pi i}$}}									(C1)
							(Gmn)	edge				node[right]		{{$\frac{\log_{n,W'}}{2\pi i}$}}									(Cn);
			\draw[arrows = {Hooks[right]-stealth}]
							(PE) -- (Lie);
 		\end{tikzpicture}
	\end{center}
	\caption{\label{fig:diagramma_mappe}These are the maps used to define the logarithm of the semialgebraic set $S$.}
\end{figure}
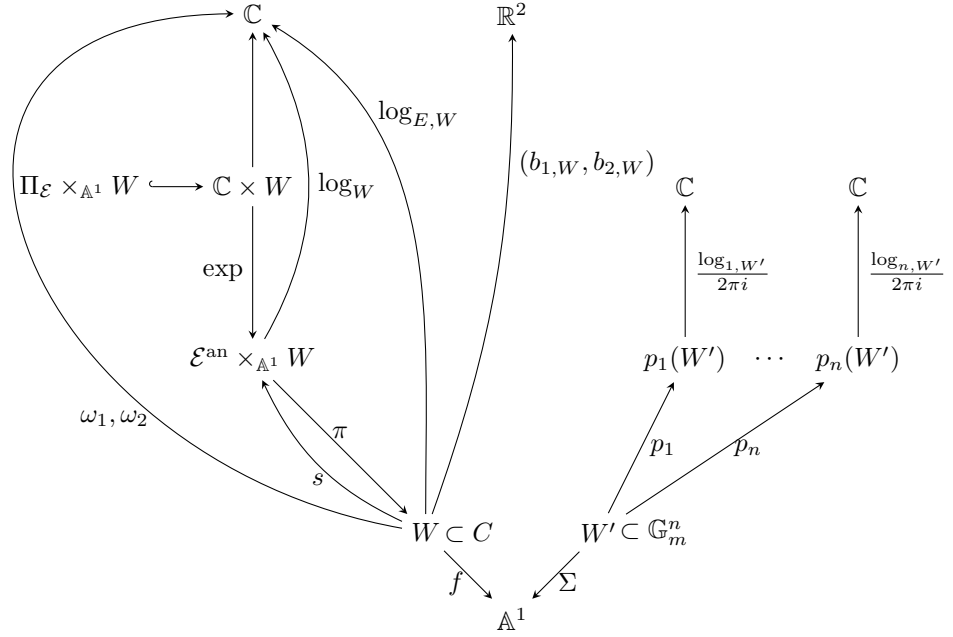

\begin{definition}
\label{defi:theta}
	For any open subset $V$ of $(C\times_{\mathbb{A}^{1}}\G_m^n)^{an}$ such that its projections $\tilde{V}$ and $\tilde{V}'$ to $C^{an}$ and ${\mathbb{G}_{m}^{n}}^{an}$ respectively are open contractible subsets, we define the map $\theta_V:V\to \mathbb{R}^{2}\times\C^n$ as 
	\begin{equation}
		\theta_V=(b_{1,\tilde{V}},b_{2,\tilde{V}},a_{1,\tilde{V}'},\ldots,a_{n,\tilde{V}'}).
	\end{equation}
\end{definition}

\begin{remark}
\label{oss:cover_di_S}
	Since $S$ is a semialgebraic compact set, we can find a finite number of open subsets $C_{1},\ldots, C_{r}$ of ${(C\times_{\mathbb{A}^{1}}\mathbb{G}_{m}^{n})}^{an}$ definable in $\mathbb{R}_{an,exp}$ such that the projections of each $C_{j}$ to $C^{an}$ and $\mathbb{G}_{m}^{n}$ are open and contractible and the union $\bigcup_{j=1}^{r}C_{j}$ contains $S$.
\end{remark}

\begin{definition}
\label{defi:Logaritmo_S}
	With notation as in the above remark, we define the \emph{logarithm} of $S$ to be the set 
	\begin{equation}
		\mathscr{S}=\bigcup_{i=1}^n \theta_{C_i}(S\cap C_i).
	\end{equation}
\end{definition}

In the following proposition we collect some properties of the set $\mathscr{S}$ which follow directly from the definition.

\begin{proposition}
	\label{prop:properties logarithm}
	$\mathscr{S}$ is definable in the o-minimal structure $\mathbb{R}_{an,exp}$. For all $(P, \bm{x} )\in S$, $\sigma(P, \bm{x})$ is torsion if and only if the corresponding point in $\mathscr{S}$ is rational.
\end{proposition}

\begin{remark}
	\label{rem:abuse of notation}
	In the above proposition we say that there is a point in $\mathscr{S}$ corresponding to each point $(P, \bm{x} )\in S$. This is not precise, there might be more than one point in $\mathscr{S}$ whose exponential is the
	 point $(P, \bm{x})$, however the set of such points is finite.
\end{remark}


\subsection{Semialgebraic Subsets of the Logarithm} 
\label{sub:semialgebraic_subsets_of_the_logarithm}

As stated before, we will apply Pila-Wilkie's Theorem \ref{thm:Pila-Wilkie} to the set $\mathscr{S}$. To do this we will need some information about the semialgebraic subsets of $\mathscr{S}$. The aim of this section is to prove the following proposition.

\begin{proposition}
\label{prop:sottoinsiemi_semialgebrici_logaritmo}
	Let $\mathscr{T}$ be a connected semialgebraic subset of $\mathscr{S}\subset \mathbb{R}^{2}\times\C^n$; then the projection of $\mathscr{T}$ to $\R^2$ is constant.
\end{proposition}

Let $\mathscr{C}$ be a semialgebraic curve in $\mathscr{S}$; we can reduce to the case in which $\mathscr{C}$ is contained in one of the $\theta_i(C_i)$. In what follows we fix $C=C_i$ and drop any subscript $C$ we used in the definition of the set $\mathscr{S}$. Moreover, up to reducing to a smaller open subset $C$, we may assume that a single determination $\log$ of the complex logarithm is defined on all projections $p_j(C)$, $j=1,\ldots, n$ and such determination is used to define the functions $a_{1},\ldots,a_{n}$.

Now let $\Gamma$ be the inverse image of $\mathscr{C}$ by $\theta$; then $\Gamma$ is a real analytic curve in $S$. Since the image of $\Gamma$ by $\theta=(a_1,\ldots,a_n,b_1,b_2)$ is the semialgebraic curve $\mathscr{C}$, we have that the transcendence degree of the field $\C(a_1|_{\Gamma},\ldots,a_n|_{\Gamma},b_1|_{\Gamma},b_2|_{\Gamma})$ over $\C$ is at most $1$. Thus also the transcendence degree of $\C(\omega_1|_{\Gamma},\omega_2|_{\Gamma}, a_1|_{\Gamma},\ldots,a_n|_{\Gamma},b_1|_{\Gamma},b_2|_{\Gamma})$ over $\C(\omega_1|_{\Gamma},\omega_2|_{\Gamma})$ is at most $1$. Moreover, as in the previous section, let $p_i:\mathbb{G}_{m}^{n}\to \mathbb{G}_{m}$ be the projection to the $i$-th factor, then, by definition, $\log \circ p_{i}=2\pi i a_i$ and $\log_{E}=(\omega_1b_1+\omega_2b_2)$, thus 
\begin{equation}
	\log\circ p_1,\ldots,\log\circ p_n,\log_{E}\in \C(\omega_1,\omega_2,a_1,\ldots,a_n,b_1,b_2).
\end{equation}
It follows that 
\begin{equation}
	\trdeg_{\C(\omega_1|_{\Gamma},\omega_2|_{\Gamma})} \C(\omega_1|_{\Gamma},\omega_2|_{\Gamma},\log\circ p_1|_{\Gamma},\ldots,\log\circ p_n|_{\Gamma},\log_{E})\leq 1.
\end{equation}
Hence we can find $n$ polynomials $P_1,\ldots,P_n$ in $\C(\omega_1,\omega_2)[X_1,\ldots,X_{n+1}]$ such that each of the $P_i$ does not belong to the radical of the ideal generated by the others and moreover the $n$ functions $\gamma_i(P,\bm{x})= P_i(\log(x_1),\ldots,\log(x_n),\log_{E}(P))$, $i=1,\ldots,n$, are holomorphic in a neighbourhood of $\Gamma$ and vanish on it. Finally we let $\Gamma_i$ be the vanishing set of $\gamma_i$. Proposition \ref{prop:sottoinsiemi_semialgebrici_logaritmo} will follow from the following Lemma.

\begin{lemma}
\label{prop:dimensione_intersezione_gamma_i}
	Assume that $\Sigma$ is not constant on $\Gamma$. Then, for any $i=1,\ldots,n$, the intersection $\Gamma_1\bigcap\ldots\bigcap \Gamma_i$ is a complex analytic variety of dimension at most $n-i$.
\end{lemma}

The fundamental part of the proof of this lemma is given by the following result.

\begin{lemma}
\label{lemm:grado_trascendenza_logaritmi}
	Let $i\in\{1,\ldots,n\}$ and $V\subseteq S$ homeomorphic to an open ball in $\C^i$, moreover let us suppose that the coordinates $x_1,\ldots,x_i$ are algebraically independent over $V$ and that the sum $x_1+\ldots+x_n$ is not constant on $V$. Then the field
	\begin{equation*}
		\C(\omega_1|_{V},\omega_2|_{V},\log_E|_{V},\log\circ p_1|_{V},\ldots,\log\circ p_i|_{V})
	\end{equation*}
	has trascendence degree (at least) $i+1$ over
	\begin{equation*}
		{\C(\omega_1|_{V},\omega_2|_{V})}. 
	\end{equation*}
\end{lemma}

The case of a constant elliptic scheme is settled by \cite[Theorem 2]{BrownawellKubota77}; we recall below a special case of this theorem which is sufficient for our purposes.

\begin{proposition}[Special case of Theorem 2 \cite{BrownawellKubota77}]
\label{prop:teo2_Kubota}
	Let $\mathscr{P}$ and $\zeta$ be Weierstrass functions associated to the pair of periods $\omega_1$, $\omega_2$. Let $y_0,\ldots,y_i\in \C\left[\left[t_1,\ldots,t_n\right]\right]$ without constant term such that $y_0\neq 0$ and $y_1,\ldots,y_i$ are linearly independent over $\Q$. Let $r$ be the rank of the jacobian matrix $(\de y_i/\de t_j)$. Then the transcendence degree over $\C$ of the field $\C(y_0,\ldots,y_i,\mathscr{P}(y_0),e^{y_1},\ldots,e^{y_i})$ is at least $i+1+r$.
\end{proposition}

Applying the above result with $y_{0}=\log_E|_{V}$ and $y_{j}=\log\circ p_{j}$, $j=1,\ldots, i$, we have $\mathscr{P}(y_{0})=\Sigma$ and $e^{y_{j}}=p_{j}$ for $j=1,\ldots, i$ and hence the result in the case of a constant elliptic curve.

\begin{proof}
	By the remark above, we may assume that the scheme $\mathcal{E}/C$ is not constant. Then we consider the following tower of differential fields: $K=\C(x_1,\ldots,x_n)$, $F=K(\omega_1,\omega_2,\eta_1,\eta_2)$ and $L=F(\log_{E},\zeta(\log_{E}))$, where $\eta_1,\eta_2$ are the semi periods and $\zeta$ is the Weierstrass zeta function associated to the elliptic scheme $\mathcal{E}/C$. It is known that the differential Galois group of the extension $F/K$ is $SL_2(\C)$. Moreover, from \cite{Bertrand89},
	we know that the differential Galois group $\mathcal{V}$ of $L$ over $F$ is isomorphic to ${\G_a}^2$. Now we observe that the Galois group of $K(\log\circ p_j)/K$ is ${\G_a}$ for any $j=1,\ldots,i$; since $\G_a$ is not a quotient of $SL_2(\C)$, we have that also the Galois group of $F(\log\circ p_j)/F$ is $\mathcal{W}_i=\G_a$.

	Since the coordinates $x_1,\ldots,x_i$ are algebraically independent, the Galois group of $F ( \log \circ p_1 , \ldots , \log\circ p_i  )$ over $F$ is $\mathcal{W}=\bigoplus_{j=1}^i\mathcal{W}_i={\G_a}^i$. Up to this point we have considered the following differential field extensions
	
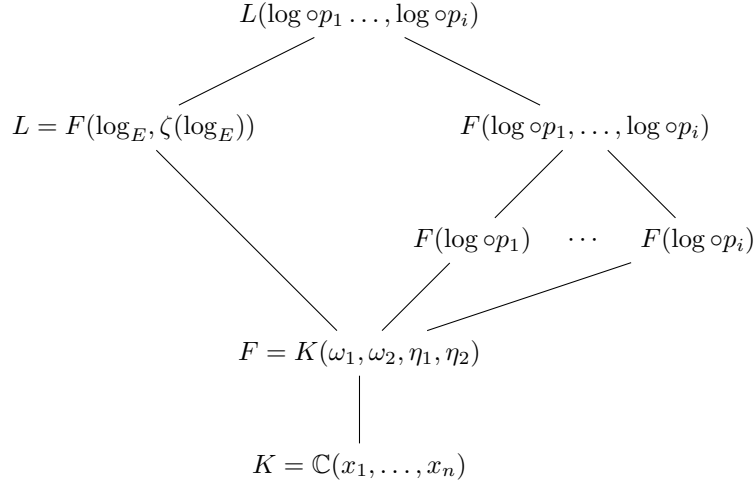
\begin{center}
		\begin{tikzpicture}[scale=1.5]
   			\node (K)		at		(0,0)		{$K=\C(x_1,\ldots,x_n )$};
   			\node (F)		at		(0,1)		{$F=K(\omega_1,\omega_2,\eta_1,\eta_2 )$};
   			\node (F1)		at		(1,2)		{$F(\log\circ p_1)$};
   			\node (pti)		at		(2,2)		{$\cdots$};
   			\node (Fi)		at		(3,2)		{$F(\log\circ p_i)$};
   			\node (L)		at		(-2,3)		{$L=F(\log_E,\zeta(\log_E))$};
   			\node (F1i)		at		(2,3)		{$F(\log\circ p_1,\dots,\log\circ p_i)$};
   			\node (L1i)		at		(0,4)		{$L(\log\circ p_1\ldots,\log\circ p_i)$};
   			\draw[-] 	(K)		edge	(F)
   						(F)		edge	(F1)
   						(F)		edge	(Fi)
   						(F)		edge	(L)
   						(F1)	edge	(F1i)
   						(Fi)	edge	(F1i)
   						(L)		edge	(L1i)
   						(F1i)	edge	(L1i);
		\end{tikzpicture}
	\end{center}
	These extensions give us the following exact sequences of differential Galois Groups.
	\begin{align}
		0\to& \mathcal{V}={\G_a}^2\to Gal(L/K)\to SL_2(\C)\to 0;\\
		0\to& \mathcal{W}_j={\G_a}\to Gal(F(\log\circ p_j)/K)\to SL_2(\C)\to 0.
	\end{align}
	Thus we have that the two abelian groups $\mathcal{V}={\G_a}^2$, $\mathcal{W}_j={\G_a}$ are $SL_2(\C)$-modules. Moreover, we know that the action of $SL_2(\C)$ is the standard one on $\mathcal{V}$ and the trivial one on each of the $\mathcal{W}_j$. This implies that the Galois group of  $L(\log(x_1),\ldots,\log(x_i))/F$ is a sub $SL_2(\C)$-module $\mathcal{T}$ of $\mathcal{V}\oplus \mathcal{W}_1\oplus\ldots\oplus\mathcal{W}_i={\G_a}^2\oplus {\G_a}^i$ that projects onto $\mathcal{V}$ and $\bigoplus_{j=1}^i \mathcal{W}_j$. The only such submodule is $\mathcal{T}=\mathcal{V}\oplus \mathcal{W}_1\oplus\ldots\oplus\mathcal{W}_i$. Thus we obtain:
	\begin{equation}
		\trdeg_{F}F(\log_{E},\zeta(\log_{E}),\log \circ p_1,\ldots,\log\circ p_i)=\dim \mathcal{T}=i+2.
	\end{equation}
	Which is exactly what we were looking for.
\end{proof}

Now we prove  Proposition \ref{prop:dimensione_intersezione_gamma_i}.

\begin{proof}[Proof of Proposition \ref{prop:dimensione_intersezione_gamma_i}]
	Let us suppose that the dimension of $\Gamma_1 \bigcap\ldots \bigcap \Gamma_{i+1}$ is strictly greater than $n-i-1$. Then we could find a subset $V$ of this intersection which is homoemorphic to an open ball in $\C^{n-i}$. Up to reordering coordinates, we can assume that $x_1,\ldots,x_{n-i}$ are algebraically independent on $V$; moreover, since the sum of coordinates is not constant on $\Gamma_1 \bigcap\ldots \bigcap \Gamma_{i+1}$, we can assume that the same holds on $V$. We can thus apply the previous Lemma and obtain that the transcendence degree of $\C(\omega_1|_{U},\omega_2|_{U},\log_E|_{U},\log\circ p_1|_{U},\ldots,\log\circ p_{n-i}|_{U})$ over $\C(\omega_1|_{U},\omega_2|_{U})$ is $n-i+1$. Finally we observe that $U\subset \Gamma_1 \bigcap\ldots \bigcap \Gamma_{i+1}$ and that this gives us $i+1$ independent algebraic relations between the functions $\log_{E}|_{U},\log(x_1)|_{U},\ldots,\log(x_{n})|_{U}$; thus the field extension analysed above would have transcendence degree at most $n-i$ which is absurd.
\end{proof}

We can now prove Proposition \ref{prop:sottoinsiemi_semialgebrici_logaritmo}.

\begin{proof}[Proof of Proposition \ref{prop:sottoinsiemi_semialgebrici_logaritmo}]
	Let $\mathscr{T}$ be a connected semialgebraic subset of dimension $\geq 1$ of $\mathscr{S}\subset \C^n\times\R^2$; if the projection of $\mathscr{T}$ on $\R^2$ was not constant, then we could find a semialgebraic curve $\mathscr{C}$ in $\mathscr{T}$ with the same property. Now, constructing $\Gamma$ and the $\Gamma_i$ as above, we would obtain that $\Gamma\subset \bigcap_{i=1}^n \Gamma_i$. Thus the intersection $\Gamma\subset \bigcap_{i=1}^n \Gamma_i$ would not have complex dimension zero as predicted by Proposition \ref{prop:dimensione_intersezione_gamma_i}.
\end{proof}


\subsection{End of the Proof} 
\label{sub:end_of_the_proof}

Given $\zeta\in \C$; define $V_{\zeta}$ to be the inverse image of $\zeta$ in $C\times_{\mathbb{A}^{1}}\G_m^n$ 
\begin{equation}
	\begin{split}
		C\times_{\mathbb{A}^{1}}\G_m^n &\to \mathbb{A}^{1}\\ 
		V_{\zeta}&\mapsto \zeta.
	\end{split}
\end{equation}

\begin{proposition}
	\label{prop:scrittura_come_somma_di_radici_di_1}
	Let $(P, \bm{\varepsilon} )\in C\times_{\mathbb{A}^{1}}\mathbb{G}_{m}^{n}$ such that $\bm{\varepsilon}\in \mathbb{G}_{m}^{n}$ is torsion. Assume that, for all $I\subset \{1,\ \ldots,\ n\}$, the sum $\sum_{i\in I}\varepsilon_{i}$ is different from zero. Let $\zeta= \Sigma(\bm{\varepsilon})$. Let $H$ be an algebraic subgroup of $\mathbb{G}_{m}^{n}$ such that ${P}\times_{\mathbb{A}^{1}} \bm{\varepsilon}H\subset V_{\zeta}$. Then $H$ has dimension zero.
\end{proposition}

\begin{proof}
	Keeping notation as in the statement of the proposition; consider the subvariety
	\begin{equation}
		\tilde{V}_{\zeta}: x_{1}+\ldots+x_{n}=\zeta
	\end{equation}
	of $\mathbb{G}_{m}^{n}$. It is sufficient to prove that if $H$ is an algebraic subgroup of $\mathbb{G}_{m}^{n}$ such that $\bm{\varepsilon}H\subset \tilde{V}_{\zeta}$ then the dimension of $H$ is zero. 

	Let $H$ be a maximal subgroup of $\mathbb{G}_{m}^{n}$ such that $\bm{\varepsilon}H\subset \tilde{V}_{\zeta}$. 
	Then $H$ is the maximal subgroup of $\mathbb{G}_{m}^{n}$ contained in $\bm{\varepsilon}^{-1}\tilde{V}_{\zeta}$. The variety $\bm{\varepsilon}^{-1}  \tilde{V}_{\zeta}$ is given by the equation $\varepsilon_{1}x_{1}+\ldots+\varepsilon_{n}x_{n}=\zeta$. Following the proof of \cite[Propostion 3.2.14]{BombieriGubler06}, rewrite this equation as
	\begin{equation}
		a_{0} \bm{x}^{\bm{\lambda}_{0} }+\ldots + a_{n}\bm{x}^{\bm{\lambda}_{n} } =0,   
	\end{equation}
	where $\bm{\lambda}_{0}$ is the zero vector, $a_{0}=-\zeta$, $\bm{\lambda}_{i} $ is the $i$-th vector of the standard basis for $\mathbb{Z}^{n}$ and $a_{i}=\varepsilon_{i}$ for $1\leq i\leq n$.

	Given $0\leq i\leq n$, the monomial $\bm{x}^{\bm{\lambda}_{i}} $ restricted to $H$ is a character of $H$ into $\mathbb{C}$, which we will denote by $\chi_{i}$. Now, given a character $\chi$ of $H$, consider the set
	\begin{equation}
		\mathcal{L}_{\chi}=\left\{ 0\leq i\leq n \ \vert\ \chi_{i}=\chi \right\}.
	\end{equation}
	Since $H$ is contained in $\bm{\varepsilon}^{-1}\tilde{V}_{\zeta} $, we have the relation
	\begin{equation}
		\sum_{\chi}\left(\sum_{i\in \mathcal{L}_{\chi}}a_{i}\right)\chi = 0.
	\end{equation}
	By \cite[Theorem 4.1, Chapter VI]{Lang02}, this linear relation must be trivial. Hence each sum
	\begin{equation}
		\sum_{i\in \mathcal{L}_{\chi}}a_{i}=0.
	\end{equation}
	Since we assumed that no sub-sum $\sum_{i\in I}\varepsilon_{i}$ is zero,
	the only possibility is that, for each character $\chi$, $\mathcal{L}_{\chi}=\left\{ 0,\ldots,n \right\}$. By definition of $\mathcal{L}_{\chi}$, we have that $\chi=\chi_{0}$ which is the trivial character. Hence the only character of $H$ is the trivial character and $H$ is the identity subgroup. 
\end{proof}

We now recall the Ax-Lindemann-Weierstrass theorem for the multiplicative group; see \cite{Orr15}.

\begin{theorem}[Ax-Lindemann-Weierstrass]
	\label{thm:Ax-Lindemann-Weierstrass}
	Let $\exp: \mathbb{C}^{n}\to \mathbb{G}_{m}^{n}$ be the exponential map. Let $\mathscr{T}\subset \mathbb{C}^{n}$ be a semialgebraic subset. Then the Zariski closure of $\exp(\mathscr{T})$ is a translate of an algebraic subgroup of $\mathbb{G}_{m}^{n}$.
\end{theorem}

As a consequence of the above results we have.

\begin{proposition}
\label{prop:sottoinsiemi_semialgebrici_V_zeta}
	Let $(P,\bm{\varepsilon})\in S\subset C\times_{\mathbb{A}^{1}}\G_m^n$ a point such that $\sigma(P,\bm{\varepsilon})$ is torsion. Assume that, for all $I\subset \{1,\ \ldots,\ n\}$ the sum $\sum_{i\in I}\varepsilon_{i}$, is different from zero. Let $(x_{1},x_{2}, \bm{z} )\in \mathscr{S}\subset \mathbb{R}^{2}\times \mathbb{C}^{n}$ a point which projects to $(P, \bm{\varepsilon} )$. Then $(x_{1},x_{2}, \bm{z} )$ does not lie in $\mathscr{S}^{alg}$.
\end{proposition}

\begin{proof}
	Assume there is a connected semialgebraic subset $\mathscr{T}\subset \mathscr{S}$ containing $(x_{1},x_{2}, \bm{z} )$. By Proposition \ref{prop:sottoinsiemi_semialgebrici_logaritmo}, the projection of $\mathscr{T}$ to $\mathbb{R}^{2}$ is constant. Let $\tilde{\mathscr{T}}$ be the projection of $\mathscr{T}$ to $\mathbb{C}^{n}$. By definition of $\mathscr{S}$, the two compositions
	\begin{equation}
		\begin{split}
			\mathscr{S}\subset \mathbb{R}^{2}&\times \mathbb{C}^{n}\to \mathbb{R}^{2}\frecciadx{\exp}{} \mathcal{E}\frecciadx{\pi}{} C \frecciadx{f}{} \mathbb{A}^{1},\\
			\mathscr{S}\subset \mathbb{R}^{2}&\times \mathbb{C}^{n}\to \mathbb{C}^{n}\frecciadx{\exp}{} \mathbb{G}_{m}^{n} \frecciadx{\Sigma}{} \mathbb{A}^{1}
		\end{split}
	\end{equation}
	are equal. Then $\exp(\tilde{\mathscr{T}})\subset \mathbb{G}_{m}^{n}$ is contained in
	\begin{equation}
		\tilde{V}_{\zeta}:\ x_{1}+\ldots+x_{n}=\zeta.
	\end{equation}
	By Theorem \ref{thm:Ax-Lindemann-Weierstrass}, the Zariski closure of $\exp{\tilde{\mathscr{T}}}$ is a translate of an algebraic subgroup, which, by what we have said above, is contained in $\tilde{V}_{\zeta}$ and contains $\bm{\varepsilon}$. 
	Since we assumed that no sub-sum $\sum_{i\in I}\varepsilon_{i}$ is zero, by Proposition \ref{prop:scrittura_come_somma_di_radici_di_1}, $\tilde{V}_{\zeta}$ does not contain any positive dimensional translate of an algebraic subgroup of $\mathbb{G}_{m}$ containing $\bm{\varepsilon}$. Hence $\mathscr{T}$ has dimension zero and, since it is connected, it reduces to a point.
\end{proof}

We can now finish the proof of our main result. Let $n\in \mathbb{N}$ and $\mathcal{L}_n$ be the set of points $(P, \bm{\varepsilon} )\in C\times_{\mathbb{A}^{1}}\mathbb{G}_{m}^{n}$ such that $\sigma(P, \bm{\varepsilon} )$ is torsion and, for all $I\subset \{1,\ \ldots,\ n\}$, the sum $\sum_{i\in I}\varepsilon_{i}$ is different from zero. Then, it is sufficient to show that the set $\mathcal{L}_n$ is finite.

First we recall that for each $(P,\bm{x})\in S$ such that $\sigma(P,\bm{x})$ is torsion, $\theta(P,\bm{x})\in \mathscr{S}$ is rational. Moreover, since $S$ is compact, $\mathscr{S}$ is bounded; thus there is some number $c_1$ such that $H(\theta(P,\bm{x}))\leq c_1 T$. If $(P,\bm{\varepsilon})\in C\times_{\mathbb{A}^{1}}\G_m^n$ is a torsion point such that $\sigma(P,\bm{\varepsilon})$ is torsion of order $T$, then, from Corollary \ref{cor:grado_punti_torsione} and from the definition of $S$, we have that there are two numbers $c_2,c_3>0$ such that $S$ contains at least $1/2\, c_2 T^{c_3}$ points with the same property. Let $0<c_4<c_3$; by Pila-Wilkie's Theorem, we have that there is a number $c_5$ such that $N(\mathscr{S}\setminus \mathscr{S}^{\alg},T)<c_5 T^{c_4}$. From what we have proven above, we have that if $(P,\bm{\varepsilon})\in S$ is in $\mathcal{L}_n$, then $\theta(P,\bm{\varepsilon})\in \mathscr{S}\setminus\mathscr{S}^{\alg}$. 

Thus we have that if there is some torsion point in $(P,\bm{\varepsilon})\in C\times_{\mathbb{A}^{1}}\G_m^n$ contained in $\mathcal{L}_n$ and such that $\sigma(P,\bm{\varepsilon})$ has order $T$, then 
\begin{equation}
	1/2\, c_2 T^{c_3}<N(\mathscr{S}\setminus \mathscr{S}^{\alg},c_1 T)<c_6 T^{c_4}
\end{equation}
where $c_6=c_5 c_1^{c_4}$. Since we have chosen $c_{3}> c_{4}$, this last inequality limits the maximum order for $\sigma(P,\bm{\varepsilon})$; since this order is at least the order of $\bm{\varepsilon}$ and $f:U\to \mathbb{A}^{1}$ is a finite map, we have proven the result.



\printbibliography[heading=bibintoc]

\end{document}